\documentclass{amsart}
\usepackage{appendix}
\usepackage[colorlinks=true,
citecolor=black,
linkcolor=black,
anchorcolor=black,
filecolor=black,
menucolor=black,
urlcolor=black,
pdftitle={On the torsion in the center conjecture},
pdfsubject={Algebraic topology},
pdfauthor={Vitali Kapovitch, Anton Petrunin and Wilderich Tuschmann}
]{hyperref}

\begin{document}
\title[On the torsion in the center conjecture]{On the\\ torsion in the center\\ conjecture} 
\rm
\author{Vitali Kapovitch}
\address{Vitali Kapovitch\\Department of Mathematics\\University of Toronto\\
Toronto, Ontario, M5S 2E4, Canada.}\email{vtk@math.toronto.edu}
\author{Anton Petrunin }\address{Anton Petrunin\\ Department of Mathematics\\ Pennsylvania State University\\
University Park, State College, PA 16802, USA.
}\email{petrunin@math.psu.edu}
\author{Wilderich Tuschmann}\address{Wilderich Tuschmann
\\Arbeitsgruppe Differentialgeometrie
\\Institut f\"ur Algebra und Geometrie
\\Fakult\"at f\"ur Mathematik
\\Karlsruher Institut f\"ur Technologie
\\Englerstr. 2
\\D-76131 Karlsruhe, Deutschland.}\email{tuschmann@kit.edu}

\begin{abstract}
We present a condition for towers of fiber bundles which implies that the fundamental group of 
the total space has a nilpotent subgroup of finite index whose torsion is contained in its center.
Moreover, the index of the subgroup can be bounded in terms of the fibers of the tower.

Our result is motivated by the conjecture that
every almost nonnegatively curved closed $m$-dimensional manifold $M$ 
admits a finite cover $\tilde M$ for which the number of leafs is bounded in terms of $m$
such that the torsion of the fundamental group $\pi_1 \tilde M$ lies in its center.
\end{abstract}

\maketitle

\section*{Introduction}

\subsection{}
We prove the following topological result, which is motivated by Conjecture~\ref{con:tor} below.

\begin{thm}\label{thm:smooth}
Let $F_1,F_2,\dots,F_n$ be an array of closed manifolds 
such that each $F_i$ is either $\mathbb{S}^{1}$ or is simply connected. 
Assume $E$ is the total space of a tower of fiber bundles over a point
$$E=E_n\buildrel {F_n}\over \longrightarrow E_{n-1}\buildrel {F_{n-1}}\over\longrightarrow\dots\buildrel {F_1}\over\longrightarrow E_0=\{pt\}$$
and each of the bundles $E_k\buildrel {F_k}\over \longrightarrow E_{k-1}$ is homotopically trivial over the $1$-skeleton. 
Then the fundamental group $\pi_1E$ contains a nilpotent subgroup $\Gamma$ such that
$$[\pi_1E:\Gamma]\le \mathrm{Const}(F_1,F_2,\dots,F_n)\quad\text{and}\quad\mathrm{Tor}(\Gamma)\subset \mathrm{Z}(\Gamma),$$
where $\mathrm{Tor}(\Gamma)$ and $\mathrm{Z}(\Gamma)$ denote the torsion and the center of $\Gamma$, respectively.
\end{thm}

Our proof is surprisingly involved. 
We would be very interested to see a proof based on a different idea;
in particular, it might help to establish the conjecture in  full generality.

\section*{Motivation}

\subsection{Conjectures} Theorem 1 is motivated by the following conjecture from \cite{KPT}.

\begin{mconj}\label{con:tor}
For all dimensions $m$ there is a constant $C=C(m)$ such that if $M^m$ is an almost nonnegatively curved 
closed smooth $m$-manifold, then there is a nilpotent subgroup $N\subset \pi_1M$ of index  at most $C$ whose torsion is contained in its center.
\end{mconj}

The following earlier conjecture of Fukaya and Yamaguchi \cite{FY} is closely related.

\begin{conj}[Fukaya--Yamaguchi]\label{con:c-ab}
The fundamental group of a nonnegatively curved  $m$-manifold $M$ is $C(m)$-abelian: 
there is $C=C(m)$ such that if $M^m$ is nonnegatively curved, then there is an 
abelian subgroup $A\subset \pi_1M$ of index  at most~$C$.
\end{conj}

Let us show that our main conjecture (if true) implies Conjecture \ref{con:c-ab}.

If the fundamental group of $M$ is finite (in particular if $M$ is positively curved), then the whole group is torsion.
In this case Conjecture~\ref{con:tor} and Theorem B in \cite{KPT} imply Conjecture~\ref{con:c-ab}.
The general case is more tricky.
The following argument
was suggested to us by Burkhard Wilking (compare to \cite[Corollary 4.6.1]{KPT}).

If $\sec(M^m)\ge0$, then the universal cover $\tilde M$ of $M$ is isometric to the product $\mathbb{R}^n\times K$, where $K$ is a compact Riemannian manifold and the $\pi_1M$ action  on $\mathbb{R}^n\times K$ is diagonal~\cite{CG72}.
It follows from \cite[Cor. 6.3]{wilking} that one can deform the metric on $M$ so that its universal cover is still isometric to $\mathbb{R}^n\times K$ and the induced action on $K$ is finite.
Let $\Gamma_1$ be the image of $\pi_1M$ in $\Iso(\R^n)$ and $\Gamma_2$ be the image in $\Iso(K)$. By the above we can assume that $\Gamma_2$ is finite.
Therefore $\Gamma_1$ is a crystallographic group.
By the third Bieberbach theorem, it contains a subgroup of index $\le C(n)$ which consists entirely of translations (and hence is abelian).
Consider the action of $\Gamma_2$ on the frame bundle $F$ of $K$.
This action is free and by Cheeger's trick, $F$ admits a sequence of almost nonnegatively curved metrics~\cite{FY}. 
Since $\Gamma_2$ is finite, by Conjecture~\ref{con:tor} it contains 
an abelian subgroup of index $\le C(m)$. 
Thus, both $\Gamma_1$ and $\Gamma_2$ are $C(m)$-abelian,
and this yields Conjecture~\ref{con:c-ab}.

\subsection{About the reduction.} Let us explain  how Theorem~\ref{thm:smooth} is related to Conjecture~\ref{con:tor}.

An almost nonnegatively curved manifold can be defined as a closed smooth manifold $M$ which admits a sequence of metrics $g_n$ with a uniform lower curvature bound such that the sequence $M_n$ converges to a point in Gromov--Hausdorff topology.

One approach to studying almost nonnegatively curved manifolds is by studying successive blow-ups of the collapsing sequence $M_n=(M,g_n)$, as was done in \cite[Section 4.3]{KPT}.
Let us describe this construction.

The sequence $\{M_n\}$ converges to a point; this one-point space will be denoted by $A_0$.

Set $M_{n,1}:=M_n$ and $\vartheta_{n,1}:=\diam M_{n,1}$.
Rescale now $M_{n,1}$ by $\tfrac{1}{\vartheta_{n,1}}$
so that 
\[\diam(\tfrac{1}{\vartheta_{n,1}}\cdot M_{n,1})=1.\]
Passing to a subsequence if necessary, one has
that the manifolds $\frac{1}{\vartheta_{n,1}}{\cdot} M_{n,1}$
converge to $A_1$,
where $A_1$
is a compact nonnegatively curved Alexandrov space with diameter $1$.

Now choose a regular point $\bar p_1\in A_1$,
and consider distance coordinates around $\bar p_1\in U_1\to \mathbb{R}^{k_1}$,
where $k_1$ is the dimension of $A_1$.
The distance functions can be lifted
to $U_{n,1}\subset \frac{1}{\vartheta_{n,1}}{\cdot} M_{n,1}$.
Denote by $M_{n,2}$ the level set of $U_{n,1}\to \mathbb{R}^{k_1}$ that corresponds to $\bar p_1$.
Clearly, $M_{n,2}$ is a compact submanifold of codimension $k_1$.
Set $\vartheta_{n,2}:=\diam M_{n,2}$.
Passing again to a subsequence, one has
that the sequence $\frac{1}{\vartheta_{n,2}}{\cdot} M_{n,2}$
converges to some  Alexandrov space $A_2$.

As before, $A_2$ is a compact nonnegatively curved Alexandrov space with diameter 1.
Set $k_2:=k_1+\dim A_2$.
If one now chooses a marked point in $M_{n,2}$,
then, as $n\to\infty$,
$\frac{1}{\vartheta_{n,2}}{\cdot} M_{n}$ converges 
in the pointed Gromov--Hausdorff topology
to $A_2\times\mathbb{R}^{k_1}$,
which is of dimension $k_2$.

We repeat this procedure until, at some step, $k_\ell=m$, where $m=\dim M$.

\medskip

As a result, we obtain a sequence $\{A_i\}$ of compact nonnegatively curved Alexandrov spaces
with diameter $1$ that satisfies
$$\dim A_i=k_i-k_{i-1} 
\quad\text{ and therefore }\quad
\sum_{i=1}^\ell\dim A_i=m.$$
We also obtain
a sequence of rescaling factors $\vartheta_{n,i}=\diam M_{n,i}$,
and  a nested sequence of submanifolds
$$\{p_n\}=M_{n,\ell}\subset \cdots \subset M_{n,2}\subset M_{n,1}=M_n,$$

Now let us assume that \emph{for all choices of $\bar p_i$}
the Alexandrov spaces $A_1,\dots, A_\ell$ are closed  Riemannian manifolds.
This condition will be called the \emph{rescaled smoothness assumption}.

Notice that his assumption definitely \emph{does not need to hold} in general;
it is only made to simplify the problem.
Foremost, it makes it possible to apply the following fibration theorem of Yamaguchi~\cite{Yam} to the successive blow-ups of the collapsing sequence under consideration.

\begin{thm}[Yamaguchi's Fibration Theorem]
Let $M_n$ be a sequence of $m$-dimensional Riemannian manifolds with sectional curvature at least $k$ 
and diameters at most $D$, which converges to a Riemannian manifold $N$ in the Gromov--Hausdorff sense.
Then for all large $n$ there exists an almost Riemannian submersion $f_n\co M_n\to N$.
In particular, $f_n$ is a fiber bundle.
\end{thm}

Let $M_n$ as above satisfies the rescaled smoothness assumption.
Then it should be true that after passing to a subsequence of $M_n$,
each $M_n$ is homeomorphic to the total space of a tower of fiber bundles
\[
M_n=E_\ell\buildrel {A_\ell}\over \longrightarrow E_{n-1}\buildrel {A_{\ell-1}}\over\longrightarrow\dots\buildrel {A_1}\over\longrightarrow E_0=\{pt\}.
\]
Moreover, by using the Lipschitz properties of the gradient flow,
it should be possible to pass to a finite cover of the total space with number of leafs 
controlled by the topology of $A_i$, so that the tower of bundles can be
further refined  to a tower of the form
\[
E=E_n\buildrel {F_n}\over \longrightarrow E_{n-1}\buildrel {F_{n-1}}\over\longrightarrow\dots\buildrel {F_1}\over\longrightarrow E_0=\{pt\}
\]
such that
\begin{enumerate}[(i)]
\item Each $F_i$ is either $\mathbb{S}^{1}$ or is simply connected. 
\item Each of the bundles $E_k\buildrel {F_k}\over \longrightarrow E_{k-1}$ is homotopically trivial over the $1$-skeleton. 
\end{enumerate}

This finally leads us to the main topological result of this paper (Theorem \ref{thm:smooth} above),
which says that if such a tower of bundles exists, then its total space satisfies the conclusion of Conjecture~\ref{con:tor}.

We have, however, not worked out the details of the reduction, 
as we do not see a geometric condition which would imply the rescaled smoothness assumption.
In fact, it might  well be possible to do the reduction even in the case if all $A_i$ are general  Alexandrov spaces  \emph{without  boundary}, but we have no idea what to do if at least one of them has nonempty boundary.

\section*{Proof of the theorem}

\subsection{Technical tools}
The proof of Theorem~\ref{thm:smooth} uses the following result of Dror,  Dwyer and Kan (\cite{DDK}).

\begin{thm}\label{thm:DDK}
Let $F$ be a finite simply connected CW complex and 
$Aut_0(F)$ be the identity component  of the 
space of all homotopy equivalences $F\to F$. Then $Aut_0(F)$ is homotopy equivalent to a CW complex with a 
finite number of cells in each dimension.
\end{thm}

Let $\star=\id_F$ be the base point of $\Aut_0(F)$.

Furthermore, we consider $\mathbb{S}^1$ as the unit circle in the complex plane, so  that $z\mapsto z^n$ defines a map $\mathbb{S}^1\to \mathbb{S}^1$ of degree $n$.

\begin{cor}\label{cor:DDK}

Given a finite simply connected CW complex $F$, there is a natural number $n=n(F)$ such that the following  holds.

Assume that 
$\{\kappa_z\colon F\to F: z\in \mathbb{S}^1\}$ is a one parameter family of homotopy equivalences with $\kappa_1=\id_F$
such that for some $N\in\mathbb{N}$ the family $\kappa^N_z=\kappa_{z^N}$ is homotopic to the constant family $\mathrm{id}_F$. 
Then $\kappa^n_z$ is homotopic to the constant family $\mathrm{id}_F$.

\end{cor}

\begin{proof}[Proof.]
We can view $\kappa$ as a loop in $\Aut_0(F)$ based at $\star$.
Since $\Aut_0(F)$ is an $H$-space, its fundamental group  
$\pi_1(E_0(F))$ is abelian. 

By Theorem~\ref{thm:DDK}, $\pi_1(\Aut_0(F))$ is finitely generated,
therefore $\Tor[\pi_1(\Aut_0(F))]$ is finite.
Hence the corollary follows.
\end{proof}

We denote by $\DD$ the closed unit disc in the complex plane.

\begin{techlem}\label{techlem}
Let $F$ be a compact simply connected manifold and
$h\colon \DD\z\times F\z\to F$ be  a
$\DD$-family of homotopy equivalences $F\to F$ such that $h(u,*)=\mathrm{id}_F$ for any
$u\in\partial \DD$. 
Consider the map $\beta_x\colon \DD\to F$ defined by $\beta_x\colon u\mapsto h(u,x)$ for a fixed point $x\in F$. 
Then $[\beta_x]\in \mathrm{Tor}(\pi_2F)$ and is independent of $x$.

Moreover, given a map $f\colon \DD\to F$ let $f_h\co \DD\to F$ be given by $f_h(u)=h(u,f(u))$.
Note that  $f_h|_{\partial \DD}=f|_{\partial \DD}$. In particular, we can view $[f_h]-[f]$ as an element of $\pi_2F$.

Then 
$[f_h]-[f]=[\beta_x]\in \mathrm{Tor}(\pi_2F)$. 
\end{techlem}

\begin{proof}[Proof.] Assume the contrary. 

The map $h$ induces a map $\mathbb{S}^2\times F\to F$ which with a  slight abuse of notation we'll still denote by $h$. 
Consider the induced map on cohomology
\[H^{*+2}(F)\overset{h^*}{\longrightarrow} H^{*+2}(\mathbb{S}^2\times F)\to H^*(F)\otimes H^2(\mathbb{S}^2)\cong  H^*(F),\]
where the second map is the projection coming from the K\"unneth isomorphism.
We'll denote the resulting map $H^{*+2}(F)\to H^*(F)$ by $D_h$. It's easy to see that it's an algebra derivation on $H^*(F)$. The same obviously holds true for $D_h^\QQ\colon H^{*+2}(F,\QQ)\to H^*(F,\QQ)$

 We claim that  $D_h^\QQ$ vanishes on $H^2(F,\QQ)$. Suppose not and and that
  for some 
$\theta\in H^2(F,\QQ)$ we have $H^0(F,\QQ)\ni D_h^\QQ\theta\not=0$. 
Note that \[D_h^\QQ\theta^k=k{\cdot} \theta^{k-1}{\cdot} D_h^\QQ\theta.\]

Therefore 
\[[\theta^{k-1}]\ne0
\quad
\Longrightarrow
\quad[\theta^{k}]\ne0.\]
Thus,  $[\theta^{k}]\not=0$ for any positive integer $k$.
In particular, $F$ has infinite dimension, a contradiction. 
Thus, $D^\QQ_h|_{H^2(F,\QQ)}\equiv 0$. Hence, $D_h|_{H^2(F)}\equiv 0$ and therefore $h^*=pr_2^*$ on $H^2(F)$ where $pr_2\colon \mathbb{S}^2\times F\to F$ is the second coordinate projection. 
This  easily implies the statement of the Lemma.
\end{proof}
\begin{rmk}
Since $[\beta_x] \in \pi_2F$ is independent of the choice of $x$, we can drop the subindex $x$ and simply denote this element by $[\beta]$.
\end{rmk} 

\subsection{Two claims}

\begin{claim}
Let $\mathbb{S}^1\to E\to B$ be an oriented $\mathbb{S}^1$-bundle over a compact manifold $B$.
Assume that the action of $\pi_1B$ 
on $\pi_2B$ is trivial. 
Then  the action of $\pi_1E$ on $\pi_2E$ is trivial.
\end{claim}

\begin{proof}[Proof.] 
Let $\gamma\in \pi_1(E,p)$ and $\alpha\in \pi_2(E,p)$. 
Let us denote by $\bar\gamma\in \pi_1(B,\bar p)$ and 
$\bar\alpha\in \pi_2(B,\bar p)$ their projection to $B$.

By our assumptions, $\bar\alpha^{\bar\gamma}=\bar\alpha$, or equivalently, there is a  map
$h\colon \mathbb{S}^2\times \mathbb{S}^1\to B$ such that 
$$h(u,v)=\bar p,\quad
h(u,*)\sim\bar\gamma\quad\text{ and}\quad
h(*,v)\sim\bar\alpha.$$

The induced bundle $\mathbb{S}^1\to E'\to \mathbb{S}^2\times \mathbb{S}^1$ is trivial. 
Therefore the $\alpha^\gamma=\alpha$ 
for any element in $\gamma\in \pi_1E'$.
\end{proof}

\begin{claim}
Let $F\xrightarrow{i} E\to B$ be a fiber bundle with 
simply connected fiber $F$ over a compact manifold $B$, 
which admits a trivialization over the $1$-skeleton of $B$. 
Assume $\pi_1B$ is almost nilpotent 
and the action of $\pi_1B$ on $\pi_2B$ is almost trivial. 
Then  the action of $\pi_1E$ on $\pi_2E$ is almost trivial.
\end{claim}

\begin{proof}[Proof.] 
By taking the pullback of $F\to E\to B$ to a finite cover of $B$ we can assume that  the action of $\pi_1B$ on $\pi_2B$ is trivial and $\pi_1B$ is nilpotent.

Let $\gamma\in \pi_1(E,p)$ and $\alpha\in \pi_2(E,p)$. 
Let us denote by $\bar\gamma\in \pi_1(B,\bar p)$ and 
$\bar\alpha\z\in \pi_2(B,\bar p)$ their projections to $B$.

By assumptions, there is a map
$h\colon \mathbb{S}^2\times \mathbb{S}^1\to B$ such that 
$$h(u,v)=\bar p,\quad
h(u,*)\sim\bar\gamma\quad\text{ and}\quad
h(*,v)\sim\bar\alpha.$$
The induced bundle $F\to E'\to \mathbb{S}^2\times \mathbb{S}^1$ admits a 
trivialization over $u\times \mathbb{S}^1$.
Let $E''$ be the  induced bundle $F\to E''\to \mathbb{S}^2\times v\subset \mathbb{S}^2\times \mathbb{S}^1$.
The  monodromy map $f\colon E''\to E''$ around $\mathbb{S}^1$
induces the  identity map on the base $\mathbb{S}^2$ 
and due the bundle being trivial  over $u\times \mathbb{S}^1$,
it  can be chosen to be the identity on $F_v$.

Consider a map $m\colon \DD\to \mathbb{S}^2$ such that 
\begin{enumerate}[(1)]
\item  $m$ is injective in the interior of $\DD$, 
\item  $m(\partial \DD)=v$.
\end{enumerate}
Fix a trivialization 
$F\times \DD$ of the induced bundle. 
Let 
$m'\colon F\times \DD\to E''$ be the correspondent mapping. 
By construction, $m'(*,x)$ is a homeomorphism  between $F\times x$ and the fiber $F_{m(x)}$.
Then the map $f\colon E''\to E''$ is completely described by the map 
$$f'\colon F\times \DD\to F\times \DD $$ 
uniquely determined  by the  identity
$m'\circ f'=f\circ m'$. 
Clearly $f'(*,x)=\mathrm{id}_{F\times x}$ 
for any $x\in\partial \DD$.

By the Technical Lemma, there is an element 
$[\beta]\in \mathrm{Tor}(\pi_2F)$ such that 
\[[i(\beta)]=\alpha^\gamma-\alpha.\]
Since the bundle is trivial over $\mathbb{S}^1$,
\[[i(n{\cdot} \beta)]=\alpha^{\gamma^n}-\alpha\] and therefore for $n(F)=|\mathrm{Tor}(\pi_2F)|$ we have 
$\alpha^{\gamma^{n(F)}}-\alpha=0$. 

Since $\pi_1B$ is nilpotent and finitely generated we have that the subgroup  \[\langle\gamma^{n(F)}|\gamma\in \pi_1B\rangle<\pi_1B\] 
has finite index.
\end{proof}

\subsection{Reliable manifolds}

\begin{defn}
A compact connected manifold $M$ is called \emph{reliable} if 
\begin{enumerate}
\item \label{reliable1} For any 
$\tau\in \mathrm{Tor}(\pi_1(M,p))$ and for any other element $\gamma\in \pi_1(M,p)$ there is a map of $h\colon \mathbb{S}^1\times \mathbb{S}^1\to M$  such that 
$$h(u,v)= p,\quad
h(u,*)\sim\gamma\quad\text{ and}\quad
h(*,v)\sim\tau.$$
\item \label{reliable2} There is a finite cover $\iota\colon \tilde T^2\z\to \mathbb{S}^1\times \mathbb{S}^1$ and a map of solid torus $f\colon \DD\times \mathbb{S}^1\z\to M$ and an isomorphism $i\colon \tilde T^2\to \partial \DD\times \mathbb{S}^1$ such that $f\circ i=h\circ\iota$.
\end{enumerate}

\end{defn}

Evidently, if $M$ is \emph{reliable} then  $\mathrm{Tor}(\pi_1M)\subset \mathrm{Z}(\pi_1M)$.

\begin{claim}
 Let $\mathbb{S}^1\to E\to B$ be an oriented $\mathbb{S}^1$-bundle such that the base  $B$ is reliable and the action of $\pi_1B$ on $\pi_2B$ is almost trivial. 
Then $E$ is reliable.
\end{claim}

\begin{proof}
Let $\gamma\in \pi_1(E,p)$ and $\tau\in \mathrm{Tor}(\pi_1(E,p))$. 
Let us denote by $\bar\gamma\in \pi_1(B,\bar p)$ and 
$\bar\tau\in \pi_1(B,\bar p)$ their projections to $B$.

From the assumptions we have 
that there is a map 
$\bar h\colon \mathbb{S}^1\times \mathbb{S}^1\to B$ such that 
$$\bar h(u,v)=\bar p,\quad
\bar h(u,*)\sim\bar\gamma\quad\text{ and}\quad
\bar h(*,v)\sim\bar\tau,$$
which satisfy the definition of a reliable manifold.

Note that since $B$ is reliable we have $[h]=0\in H_2(B,\mathbb{Q})$. Therefore the induced bundle $\mathbb{S}^1\to E'\to \mathbb{S}^1\times \mathbb{S}^1$ is trivial, in particular $E'=T^3$. This  implies that $\gamma$ and $\tau$ commute. Hence \eqref{reliable1} is satisfied and we only need to check \eqref{reliable2}.

Since the bundle  $\mathbb{S}^1\to E'\to \mathbb{S}^1\times \mathbb{S}^1$ is trivial there is a lifting $h\colon \mathbb{S}^1\times \mathbb{S}^1\to E$ of the map $\bar h$ such that
$$ h(u,v)= p,\quad
 h(u,*)\sim\gamma\quad\text{ and}\quad
 h(*,v)\sim\tau,$$
Consider the covering map $h\circ\iota\colon \tilde T^2\to E$.

By the definition of a reliable manifold there is a covering $\iota\colon \tilde T^2\z\to \mathbb{S}^1\times \mathbb{S}^1$ and a map of solid torus $\bar f\colon \DD\times \mathbb{S}^1\z\to B$ and an isomorphism $i\colon \tilde T^2\to \partial \DD\times \mathbb{S}^1$ such that $\bar f\circ i=\bar h\circ\iota$. The pullback bundle $\mathbb{S}^1\to \tilde E'\to \DD\times \mathbb{S}^1$ is trivial. Let $s\co  \DD\times \mathbb{S}^1\to  \tilde E'$ be a section.

From above, we have that $h\circ\iota$ is homotopic to a map
$h'\colon \tilde T^2\to E$ whose projection to $B$ is the central $\mathbb{S}^1$ of 
$\DD\times \mathbb{S}^1$. 
Therefore there are loops $\psi\subset B$ and $\phi$ in a fiber $\mathbb{S}^1$  such that $h\circ\iota$ is homotopic to a torus with meridians in $s(\psi)$ and $\phi$.

Suppose 
$[\psi]$ has finite order in $\pi_1B$. By changing the cover $\tilde T\to T$ we can assume that $[\psi]=0$  in $\pi_1B$.
Then we can homotope $h'$ to a single $\mathbb{S}^1$-fiber, and this easily gives the needed map of  a solid torus.

Now suppose that $[\psi]$ has infinite order in $\pi_1B$. 
Let $\tilde\tau$ be the homotopy class of a component in the preimage of $\tau$ under the covering $\tilde T\to T$. 
On the level of  the fundamental groups we have that $\tilde\tau=[\phi]^a{\cdot} [s(\psi)]^b$ for some integer $a,b$. Projecting to $B$ we get that $\bar\tau=[\psi]^b$. Since $\tau$ is by assumption torsion and $[\psi]$ is not this means that $b=0$. Thus $\tilde\tau=[\phi]^a$ and hence
$i(\phi)\in \mathrm{Tor}(\pi_1E)$. 
Therefore passing to a finite cover of $\tilde T^2$ 
if necessary we get a map of the 
disc $\DD\to E$ which contracts one of the generators ($k \phi$) of $\tilde T^2$.
The projection of this disc to $B$ gives an element $\alpha\in \pi_2B$ 
and since $\pi_1B$ acting on $\pi_2B$ almost trivially we get 
that passing to finite cover of $\tilde T^2$ again (along $\psi$) we have a map of $\bar g\colon \mathbb{S}^2\times \mathbb{S}^1\to B$ with
$$ 
 \bar g(u,*)\sim\psi^l\quad\text{ and}\quad
 \bar g(*,v)\sim\alpha.$$
We can lift this map to a map $g\colon \DD\times \mathbb{S}^1\to E$ in such a way that $g\circ i=h'\circ \iota$. 
Together with the homotopy above, the latter gives the needed map of 
$\DD\times \mathbb{S}^1\z\to E$.
\end{proof}

\begin{claim} 
There is a positive integer $n=n(F,\dim B)$ such that the following holds.

Assume $F\to E\to B$ 
is a fiber bundle such that
\begin{enumerate}[(1)]
\item the fiber $F$ is simply connected;
\item the base $B$ is reliable;
\item the fundamental group of $B$ is nilpotent and is generated by $\le c(\dim B)$ elements;
\item the bundle admits a trivialization over the $1$-skeleton of $B$.
\end{enumerate}
Then there is an at most $n$-fold cover $\tilde E$ of the  total space $E$ which is reliable.
\end{claim}

\begin{proof} The proof is divided into two steps.

\noindent\textit{Step 1.} 
We will use the Corollary~\ref{cor:DDK} 
to construct 
an at most $n=n(F,\dim B)$-fold cover $\widetilde B\to B$ 
such that for any two elements $\bar\gamma\in \pi_1(\tilde B,\bar p)$ and $\bar\tau\in \Tor \pi_1(\tilde B,\bar p)$ there is a map 
$\bar h\colon \mathbb{S}^1\times \mathbb{S}^1\to \tilde B$ such that 
$$\bar h(u,v)=\bar p,\quad
\bar h(u,*)\sim\bar\gamma\quad\text{ and}\quad
\bar h(*,v)\sim\bar\tau,$$
which satisfy the condition in the  definition of reliable manifold and such that the induced bundle $\bar h^*$ over $\mathbb{S}^1\times \mathbb{S}^1$ with  fiber $F$ is homotopically trivial.

Let $\gamma\in \pi_1(E,p)$ and $\tau\in \mathrm{Tor}(\pi_1(E,p))$. 
Let us denote by $\bar\gamma\in \pi_1(B,\bar p)$ and 
$\bar\tau\in \pi_1(B,\bar p)$ their projections to $B$.

From the assumptions we have 
that there is a map 
$\bar h\colon \mathbb{S}^1\times \mathbb{S}^1\to B$ such that 
$$\bar h(u,v)=\bar p,\quad
\bar h(u,*)\sim\bar\gamma\quad\text{ and}\quad
\bar h(*,v)\sim\bar\tau,$$
which satisfy condition \eqref{reliable1} in the definition of a reliable manifold.

The  induced bundle $f^*$ over $\DD\times \mathbb{S}^1$ 
is trivial since the bundle over $\mathbb{S}^1$ is. 
Consider the  induced bundle $\bar h^*$ over $\mathbb{S}^1\times \mathbb{S}^1$. 
Pulling back by the finite cover $\iota\colon \tilde T^2\to \mathbb{S}^1\times \mathbb{S}^1$ the induced bundle 
$(\bar h \circ \iota)^*$ over $\tilde T^2$ becomes trivial, since it is also can be induced vial pull back by $i\colon \tilde T^2\to \DD\times \mathbb{S}^1$ and the bundle over $ \DD\times \mathbb{S}^1$ is trivial. 
The bundle over $\mathbb{S}^1\times \mathbb{S}^1$ admits a trivialization on 1-skeleton 
for the standard  product CW-structure. 
Therefore it can be induced from a bundle on $\mathbb{S}^2$ and a map of degree 1 
which sends the  1-skeleton of $\mathbb{S}^1\times \mathbb{S}^1$ to a single point of $\mathbb{S}^2$. 
Such a bundle over $\mathbb{S}^2$ can be described by its clutching map which is a one-parameter family of homeomorphisms 
$\kappa\colon \mathbb{S}^1\times F\to F$.

Note that the  homotopy type of $\kappa$ does not depend on the choice of a
trivialization over the 1-skeleton of $\mathbb{S}^1\times \mathbb{S}^1$.
Since the induced bundle over $\tilde T^2$ is trivial we have that 
if $ \iota\colon  \tilde T^2\to \mathbb{S}^1\times \mathbb{S}^1$ is an $N$-fold cover then 
$\kappa^N\colon \mathbb{S}^1\times F\to F$ is homotopic to the map $(z,x)\mapsto x$.
By Corollary \ref{cor:DDK},  $\kappa^{n(F)}$ 
is homotopic to the constant map $(z,x)\mapsto x$.

Since $\pi_1 B$ is nilpotent and generated by $\le c(\dim B)$ elements,
it follows that there is  an at most $n'=n'(F,\dim B)$-fold  cover $\widetilde B$ of $B$ such that for any two elements $\bar\gamma\in \pi_1(\tilde B,\bar p)$ and $\bar\tau\in \pi_1(\tilde B,\bar p)$ there is a map 
$\bar h\colon \mathbb{S}^1\times \mathbb{S}^1\to B$ such that 
$$\bar h(u,v)=\bar p,\quad
\bar h(u,*)\sim\bar\gamma\quad\text{ and}\quad
\bar h(*,v)\sim\bar\tau,$$
and such that the induced bundle $\bar h^*$ over $\mathbb{S}^1\times \mathbb{S}^1$ and fiber $F$ is homotopy trivial.

\medskip

From now on we will assume that that $B=\tilde B$.

\medskip

\noindent\textit{Step 2.} 
Now, using Technical Lemma~\ref{techlem}, 
we will prove that a section over 
$\bar h\colon \mathbb{S}^1\z\times \mathbb{S}^1\to B$, after finite cover (which is in fact a cover of $\tilde T^2$) bounds a solid torus;
that is, $ E$ is reliable. 

Let us summarize what we have:

\begin{enumerate}
\item \label{triv1} The  induced bundle $f^*$ over $\DD\times \mathbb{S}^1$ is trivial since the bundle over $\mathbb{S}^1$ is. 
\item \label{triv2} There is yet another  trivialization over $\partial \DD\times \mathbb{S}^1=\widetilde T^2$ which is induced from a  trivialization on $T^2=\mathbb{S}^1\times \mathbb{S}^1$.
\end{enumerate}

We can assume  that the bundle over $\DD\times \mathbb{S}^1$ is smooth and that
 both of the above trivializations are smooth. 
Consider the vector field $\bar V=\frac {\partial }{\partial t}$ on $\DD\times \mathbb{S}^1$;
here $t$ is the $\mathbb{S}^1$-coordinate and we assume $\mathbb{S}^1=\mathbb{R}/\mathbb{Z}$.

Let $p\co E'\to \DD\times \mathbb{S}^1$ be the pullback bundle over  $\DD\times \mathbb{S}^1$. 

Let $ V$ be the horizontal lift of $\bar V$ over $\tilde T^2$ with respect to  trivialization \eqref{triv2} above. 
Using partition of unity we can extend it to a vector field (which we'll still denote by $V$) on $E'$  which projects to $\bar V$. 
Let $\phi_t\co E'\to E'$ be the time $t$ integral flow of $V$. 
Note that by construction 
\begin{equation}\label{monodr-bry}
\phi_1(q)=q \text{ for any } q \in p^{-1}(\partial \DD\times\mathbb{S}^1).
\end{equation}

Since $E'$ is a trivial bundle we can identify its total space with $\DD\times F\times \mathbb{S}^1$ via, say, trivialization~\eqref{triv1} above. 
Consider the map  $m\co \DD\times F\times \mathbb{R}\to E$ given by the formula 
\[(u,x,t)\mapsto \phi_t(u,x,[0]),\]
here $u\in \DD$, $x\in F$ and $[0]\in \mathbb{S}^1=\mathbb{R}/\mathbb{Z}$.

By \eqref{monodr-bry} it follows that  $m(u,x,n)=m(u,x,0)$ for any $n\in\mathbb{Z}$, $u\in \partial \DD$ and $x\in F$.

Note that there is a map $\hat h\co \DD\times F\to F$ 
for which the following identity holds
\[m(u,\hat{h}(u,x),0)= m(u,x,1).\] 
Moreover $\hat h(u,x)=x$ for any $u\in\partial \DD$ and $x\in F$.
This means that we can apply Technical Lemma (\ref{techlem}) to $\hat h$.

Consider a horizontal section of $E'$ over $\partial \DD\times \mathbb{S}^1$ with respect to ~\eqref{triv2}. 
Fix $w\in \mathbb{S}^1$. 
This gives a with respect to ~\eqref{triv2} horizontal section over $\partial \DD\times w$.
Since 
 $F$ is simply connected this section can be extended to a section over $\DD\times w$ which in trivialization~\eqref{triv1} has the form $u\mapsto (u, w, f(u))$ for some map $f\co \DD\to F$.
As in Technical Lemma (\ref{techlem}) set  $f_{\hat h}(u)= \hat h(u,f(u))$.

Then , by Technical Lemma (\ref{techlem})
\begin{equation}
\label{eq:in-Tor}
[f_{\hat h}]-[f]=[\beta]\in \mathrm{Tor}(\pi_2F)
\end{equation}

Consider the sequence of maps $f^n\colon \DD\to F$ defined recursively, $f^0=f$ and $f^{n+1}=f^n_{\hat h}$.
By (\ref{eq:in-Tor}),
\begin{equation}
\label{eq:in-Tor-n}
[f^n]-[f^0]=n{\cdot}[\beta]=0
\end{equation}
for some positive integer $n$.

Note that 
$m(u,f^0(u),n)=m(u,f^n(u),0)$.
In particular, by (\ref{eq:in-Tor-n}),
the two maps $\DD\to E$ given by $u\mapsto m(u,f(u),0)$ and $u\mapsto m(u,f(u),n)$ are homotopic rel. to $\partial \DD$;
denote by $s_t$, $t\in[0,1]$ the homotopy.
It follows that after taking an $n$-fold cover of $\widetilde T^2$ it has a section which can be contracted by a solid torus;
the map of the solid torus is a concatenation of $(u,t)\mapsto m(u,f(u),t)$ for $t\in[0,n]$ 
and $(u,t)\to s_t(u)$ for $t\in[0,1]$. 
\end{proof}

\small
\bibliographystyle{alpha}

\end{document}